\documentclass[a4paper]{amsart}

\usepackage[T1]{fontenc}
\usepackage{lmodern}
\usepackage{amssymb, amsmath, amsthm}

\theoremstyle{plain}
\newtheorem{thm}{Theorem}
\newtheorem{prop}{Proposition}
\newtheorem{lem}{Lemma}

\theoremstyle{remark}
\newtheorem{rem}{Remark}
\newtheorem*{exam}{Example}
\newtheorem*{prob}{Problem}

\newcommand{\imax}{\underline{m}}
\DeclareMathOperator{\Arg}{Arg}

\title[smooth solutions of quasianalytic or ultraholomorphic equations]{smooth solutions of quasianalytic or ultraholomorphic equations}

\author{Vincent Thilliez}

\address{Math\'ematiques - B\^atiment M2\\
Universit\'e des Sciences et Technologies de Lille\\
F-59655 Villeneuve d'Ascq Cedex, France}

\email{thilliez@math.univ-lille1.fr}


\begin{document}

\begin{abstract} In the first part of this work, we consider a polynomial 
$ \varphi(x,y)=y^d+a_1(x)y^{d-1}+\cdots+a_d(x) $
whose coefficients  $ a_j $ belong to a Denjoy-Carleman quasianalytic local ring $ \mathcal{E}_1(M) $. Assuming that $ \mathcal{E}_1(M) $ is stable under derivation, we show that if $ h $ is a germ of $ C^\infty $ function such that $ \varphi(x,h(x))=0 $, then $ h $ belongs to $ \mathcal{E}_1(M) $. This extends a well-known fact about real-analytic functions. We also show that the result fails in general for non-quasianalytic ultradifferentiable local rings. In the second part of the paper, we study a similar problem in the framework of ultraholomorphic functions on sectors of the Riemann surface of the logarithm. We obtain a result that includes suitable non-quasianalytic situations. 
\end{abstract}

\maketitle

\section*{Introduction}

The starting point of the present paper is the following classical result:

\begin{thm}[\cite{M, Sic}]\label{classic}
Let $ \varphi $ be a non-zero germ of real-analytic function at the origin in $ {\bf R}^{n+1} $. If $ h $ is germ of $ C^\infty $ function at the origin in $ {\bf R}^n $ such that $ \varphi(x,h(x))=0 $, then $ h $ is real-analytic. 
\end{thm}

It is natural to ask whether a similar result holds in quasianalytic  situations, that is, when $ \varphi $ belongs to a quasianalytic local ring $ \mathcal{E}_{n+1}(M) $ (see the definition in Section \ref{udg}). However, the problem seems difficult to address in full generality. In Section \ref{uqac} of the present paper, we consider the particular case of quasianalytic polynomials in two variables, that is,
\begin{equation*}
\varphi(x,y)=y^d+a_1(x)y^{d-1}+\cdots+a_d(x)
\end{equation*}
where the coefficients $ a_j $ are function germs of one variable belonging to a given quasianalytic local ring $ \mathcal{E}_1(M) $. 
Despite its simplicity, this particular case is important, since in the analytic setting, Theorem \ref{classic} can be reduced to the same situation \emph{via} the preparation theorem and a potential-theoretic argument, so that the analyticity of smooth roots can then be obtained by elementary power series techniques (see e.g. Chapter V of \cite{Ba}). In the quasianalytic setting, power series expansions of germs are merely formal and cannot be used in the same way. We therefore combine an application of the Artin-type approximation property of Rotthaus \cite{Ro} together with a version of Puiseux's theorem for quasianalytic polynomials. This allows us, under a mild extra assumption, namely the stability of $ \mathcal{E}_1(M) $  under derivation, to show that if $ h $ is a germ of $ C^\infty $ function such that $ \varphi(x,h(x))=0 $, then $ h $ belongs to $ \mathcal{E}_1(M) $, as expected: this is the statement of Theorem \ref{main}. We then give an example showing that the result is generally false when $ \mathcal{E}_1(M) $ is non-quasianalytic. 

This example suggests another question, namely whether it is possible to obtain a result similar to Theorem \ref{main} in a somewhat different non-quasianaly\-tic situation, namely for ultraholomorphic classes. Thus, in Section \ref{Ult}, we consider the space $ \mathcal{A}^\infty(S) $ of functions holomorphic in a given bounded angular sector $ S $ of the Riemann surface of the logarithm and uniformly bounded in $ S $ at any order of derivation, which ensures the existence of an asymptotic expansion at the vertex. The ultraholomorphic classes $ \mathcal{A}_M(S) $ are subspaces of $ \mathcal{A}^\infty(S) $ defined by Denjoy-Carleman type estimates associated with a given sequence $ M $. Given a polynomial $ \varphi(z,w)=w^d+a_1(z)w^{d-1}+\cdots+a_d(z) $ whose coefficients $ a_j $ belong to $ \mathcal{A}_M(S) $ and an element $ h $ of $ \mathcal{A}^\infty(S) $ such that $ \varphi(z,h(z))=0 $, we show in Theorem \ref{main2} that $ h $ actually belongs to $ \mathcal{A}_M(S') $ for any strict subsector $ S' $ of $ S $, provided the sequence $ M $ satisfies the so-called \emph{moderate growth} and \emph{strong non-quasianalyticity} assumptions (see Section \ref{spos} for the definitions). For instance, the classes of functions with Gevrey expansions used in the asymptotic theory of ordinary differential equations fall within this framework.  It is also known that such a class $ \mathcal{A}_M(S) $ is non-quasianalytic if the aperture of $ S $ is sufficiently small. Thus, Theorem \ref{main2} works in certain quasianalytic and non-quasianalytic situations as well, whereas the important role of the moderate growth property of $ M $ is discussed at the end of the paper.  

\section{The Ultradifferentiable Case}\label{uqac}

\subsection{Some properties of sequences}\label{spos}
Throughout the paper, $ M=(M_j)_{j\geq 0} $ will denote a sequence of real numbers satisfying the following assumptions \eqref{norm} and \eqref{logc}:
\begin{equation}\label{norm}
\text{ the sequence } M \text{ is increasing, with } M_0=1,
\end{equation}
\begin{equation}\label{logc}
\text{ the sequence } M \text{ is \emph{logarithmically convex}}.
\end{equation}
Property \eqref{logc} amounts to saying that $ M_{j+1}/M_j $ increases. Several other assumptions on $ M $ will be considered, depending on the context. These properties, whose role will be discussed when needed, are as follows: 
\begin{itemize}
\item the so-called \emph{derivability} condition
\begin{equation}\label{stabder}
\sup_{j\geq 1}(M_{j+1}/M_j)^{1/j}<\infty,
\end{equation}
\item the obviously stronger \emph{moderate growth} condition
\begin{equation}\label{modg}
\sup_{j+k\geq 1}\left(\frac{M_{j+k}}{M_jM_k}\right)^{\frac{1}{j+k}}<\infty,
\end{equation}
\item the well-known \emph{Denjoy-Carleman quasianalyticity condition}
\begin{equation}\label{dcqa}
\sum_{j\geq 0}\frac{M_j}{(j+1)M_{j+1}}=\infty,
\end{equation}
\item the so-called \emph{strong non-quasianalyticity condition}
\begin{equation}\label{snqa}
\sup_{k\in\mathbb{N}}\frac{M_{k+1}}{M_k}\sum_{j\geq k}\frac{M_j}{(j+1)M_{j+1}}<\infty.
\end{equation}
\end{itemize}

\begin{exam} 
Being given real numbers $ \alpha\geq 0 $ and $ \beta\geq 0 $, put $ M_j=j!^\alpha\big(\ln(j+e)\big)^{\beta j} $. The sequence $ M $ satisfies properties \eqref{norm}, \eqref{logc} and \eqref{modg}. It satisfies \eqref{dcqa} if and only if $ \alpha=0 $ and $ \beta \leq 1 $. It satisfies \eqref{snqa} if and only if $ \alpha>0 $. This is the case, in particular, for Gevrey sequences $ M_j=j!^\alpha  $ with $ \alpha>0 $.
\end{exam}

With every sequence $ M $ satisfying \eqref{norm} and \eqref{logc} we also associate the function $ h_M $ defined by $ h_M(t)=\inf_{j\geq 0}t^jM_j $ for any real $ t>0 $, and $ h_M(0)=0 $. It is easy to 
check that if the sequence $ M $ also satisfies \eqref{stabder}, then for any real $ \nu>0 $, there is a constant $ C(\nu)>0 $ such that
\begin{equation}\label{hfunct1}
t^{-\nu} h_M(t)\leq C(\nu)h_M(t)\textrm{ for any } t>0.
\end{equation}
Remark also 
that for any real number $ s\geq 1 $, one obviously has 
$ \big(h_M(t)\big)^s\leq h_M(t) $. An elementary but important consequence
of the moderate growth assumption \eqref{modg} is that it implies the 
existence of a constant $ \rho(s)\geq 1 $, depending
only on $ s $ and $ M $, such that, conversely,
\begin{equation}\label{hfunct2}
h_M(t)\leq \big(h_M(\rho(s)t)\big)^s\text{ for any }t\geq 0.
\end{equation}
We refer e.g. to \cite{CC1} for a proof of the implication \eqref{modg}$\Rightarrow$\eqref{hfunct2}.

\subsection{Notation} For any multi-index $ J=(j_1,\ldots,j_n) $ of $ \mathbb{N}^n $, we denote the length $ j_1+\cdots+j_n $ of $ J $ by the corresponding lower case letter $ j $. We put $ J!=j_1!\cdots j_n! $,
$ D^J=\partial^j/\partial x_1^{j_1}\cdots\partial x_n^{j_n} $ and $ x^J=x_1^{j_1}\cdots x_n^{j_n} $. 

The order of a formal power series $ G $, that is, the lowest degree of non-zero monomials of $G $, will be denoted by $ \omega(G) $, with the usual convention $ \omega(0)=+\infty $. With any germ $ g $ of $ C^\infty $ function at the origin in $ \mathbb{R}^n $ we associate its formal Taylor expansion at the origin, hereafter denoted by $ \widehat{g} $. The order $ \omega(g) $ of $ g $ is defined as $ \omega(\widehat{g}) $.

\subsection{Ultradifferentiable function germs}\label{udg}

We recall several basic facts on quasianalytic local rings that will be needed in what follows. For a more detailed account, we refer the reader to \cite{Th}, for instance.

Let $ M $ be a real sequence satisfying \eqref{norm} and \eqref{logc}. We denote by $ \mathcal{E}_n(M) $ the set of germs $ f $ of $ C^\infty $ functions at the origin in $ \mathbb{R}^n $ for which there exist a neighborhood $ U $ of $ 0 $ in $ \mathbb{R}^n $ and positive constants $ C $ and $ \sigma $ such that
\begin{equation}
\big\vert D^Jf(x)\big\vert \leq C\sigma^jj!M_j\ \textrm{ for any } J\in \mathbb{N}^n\textrm{ and } x\in U.
\end{equation}
As explained in \cite{Th}, the set $ \mathcal{E}_n(M) $ is a local ring, with maximal ideal $ \imax_M=\{f\in\mathcal{E}_n(M) : f(0)=0\} $. It is stable under composition in the sense that any element of $ \mathcal{E}_p(M) $ operates on $ (\mathcal{E}_n(M))^p$. Since the implicit function theorem also holds in this setting, a standard argument shows that $ \mathcal{E}_n(M) $ is henselian.

It is known that the local ring $ \mathcal{E}_n(M) $ is stable under derivation if and only if $ M $ satisfies \eqref{stabder}. In this case, it is easy to see that $ \imax_M $ is generated by the coordinate functions $ x_1,\ldots,x_n $. Using the Taylor formula with integral remainder, one also checks that stability under derivation implies stability under monomial division in the following sense: if $ f $ belongs to $ \mathcal{E}_n(M) $ and if a given monomial $ x^J $ divides $ \widehat{f}$ in the ring of formal power series, then $ f(x)=x^Jg(x) $ for some $ g $ belonging to $ \mathcal{E}_n(M) $.  

A $ C^\infty$ function germ is said to be \emph{flat} if it vanishes, together with all its derivatives, at the origin. The local ring $ \mathcal{E}_n(M) $ is said to be \emph{quasianalytic} if the only flat germ that it contains is $ 0 $. By the famous Denjoy-Carleman theorem, $ \mathcal{E}_n(M) $ is quasianalytic if and only if the sequence $ M $ satisfies \eqref{dcqa}. 

When $ M $ satisfies \eqref{stabder} and \eqref{dcqa}, in other words, the conditions for quasianalyticity and stability under derivation, it is easy to check that the corresponding ring of one variable germs $ \mathcal{E}_1(M) $ is a discrete valuation ring. By \cite{Groth}, scholie 7.8.3 (iii), it is an excellent ring. As observed in \cite{Pier}, one can therefore invoke Theorem 4.2 of \cite{Ro} to get the following technical tool. 

\begin{prop}\label{Artin}
Assume that $ \mathcal{E}_1(M) $ is quasianalytic and stable under derivation. Let $ \varphi $ be an element of $ \mathcal{E}_1(M)[y] $ and let $ Z $ be a formal power series in one variable such that $ \widehat{\varphi}(x,Z(x))=0 $. Then, for every integer $ \nu\geq 0 $, there is an element $ z_\nu $ of $ \mathcal{E}_1(M) $ such that 
$ \varphi(x,z_\nu(x))=0 $ and $ \omega(Z-\widehat{z_\nu})\geq\nu $.
\end{prop}

The version of Puiseux's theorem stated hereafter can be proved by an immediate adaptation of  \cite{BM}, Section 3, since the proof given there in the analytic case only requires stability under monomial division and under composition with analytic maps, and does not use power series expansions.

\begin{prop}\label{Puiseux} 
Assume that $ \mathcal{E}_1(M) $ is quasianalytic and stable under derivation, and consider a polynomial 
$\varphi(x,y)=y^d+a_1(x)y^{d-1}+\cdots+a_d(x)$
with $ a_j\in\mathcal{E}_1(M) $ for $ j=1,\ldots,d $.
Then there exist positive integers $ m, n_1,\ldots,n_J $ and elements $ y_1,\ldots,y_J $ of $ \mathcal{E}_1(M) $ such that 
\begin{equation}\label{puiform}
\varphi(x^m,y)=\prod_{j=1}^J(y-y_j(x))^{n_j}.
\end{equation}
\end{prop}

\subsection{Smooth roots of quasianalytic polynomials}

We can now state the first main result of this paper.

\begin{thm}\label{main} Assume that $ \mathcal{E}_1(M) $ is quasianalytic and stable under derivation, and
consider a polynomial 
\begin{equation*}
\varphi(x,y)=y^d+a_1(x)y^{d-1}+\cdots+a_d(x)
\end{equation*}
with $ a_j\in\mathcal{E}_1(M) $ for $ j=1,\ldots,d $. If $ h $ is a germ of $ C^\infty $ function at the origin in $ \mathbb{R} $ such that $ \varphi(x,h(x))=0 $, then $ h $ belongs to $ \mathcal{E}_1(M) $.
\end{thm}

\begin{proof} 
Using Proposition \ref{Puiseux}, we remark first that the zero set of $ (x,y)\mapsto \varphi(x^m,y)$ is the union of the germs of smooth curves $ \Gamma_j=\{(x,y): y=y_j(x)\} $, $ j=1,\ldots, J $. Quasianalyticity implies the finiteness of $ \omega(y_i-y_j) $ for $ i\neq j $ (in particular, the $ \Gamma_j $'s do not intersect at points $ (x,y) $ with $ x\neq 0 $). Since, by assumption, we have $ \varphi(x^m,h(x^m))=0 $, we derive that there is an index $ j_0 $ such that
\begin{equation}\label{hy}
h(x^m)=y_{j_0}(x).
\end{equation}
Now, remark that the equation
\begin{equation}\label{formeq}
\widehat{\varphi}(x,Z(x))=0,
\end{equation}
where $ Z $ is a formal power series, has a finite number of solutions $ Z_1,\ldots,Z_K $: indeed, \eqref{puiform} and \eqref{formeq} imply $ Z(x^m)=\widehat{y_j}(x^m) $ for some $ j $. 
Let $ \nu$ be an integer such that $ \omega(Z_k-Z_l)<\nu $ for $ k\neq l $. Using Proposition \ref{Artin}, we obtain, for each integer $ k $, an element $ z_k=z_{k,\nu} $ of $ \mathcal{E}_1(M) $ such that
\begin{equation}\label{sol1}
\varphi(x,z_k(x))=0 
\end{equation}
and 
\begin{equation}\label{sol2}
\omega(Z_k-\widehat{z_k})\geq \nu . 
\end{equation}
From \eqref{sol1} we derive that $ \widehat{z_k} $ belongs to the finite family
$ \{Z_1,\ldots,Z_K\} $. By \eqref{sol2} and the choice of $ \nu $, we necessarily have, in fact, $ \widehat{z_k}=Z_k $ for $ k=1,\ldots,K $. Since the assumption on $ h $ implies that $ \widehat{h} $ is a solution of \eqref{formeq}, we get
\begin{equation}\label{hz}
\widehat{h}=\widehat{z_{k_0}}
\end{equation}
for some index $ k_0 $. Setting $ w(x)=z_{k_0}(x^m) $, and taking \eqref{hy} and \eqref{hz} into account, we see that $ \widehat{y_{j_0}}=\widehat{w} $. As both germs $ y_{j_0} $ and $ w $ belong to $ \mathcal{E}_1(M) $, we derive $ y_{j_0}=w $ by quasianalyticity. We therefore have proved $ h(x^m)=z_{k_0}(x^m) $, where $ z_{k_0} $ belongs to $ \mathcal{E}_1(M) $. If $ m $ is odd, we derive $ h=z_{k_0} $ and the job is done. If $ m $ is even, we also consider the germs $ \check{\varphi} $ and $ \check{h} $ defined by $ \check{\varphi}(x,y)=\varphi(-x,y) $ and $ \check{h}(x)=h(-x) $, so that $ \check{\varphi}(x,\check{h}(x))=0 $. The preceding argument yields an integer $ p\geq 1 $ and an element $ z $ of $ \mathcal{E}_1(M) $ such that $ \check{h}(x^p)=z(x^p) $, that is, $ h(-x^p)=z(x^p) $. Thus, $ h(t) $ coincides with $ z(t) $ for $ t\leq 0 $ and with $ z_{k_0}(t) $ for $ t\geq 0 $, as can be seen by setting $ t=-x^p $ and $ t=x^m $, respectively. Since $ h $ is smooth and both $ z $ and $ z_{k_0} $ belong to $ \mathcal{E}_1(M) $, we derive that $z=z_{k_0} $ and that $ h $ actually belongs to $ \mathcal{E}_1(M) $. The proof is complete.
\end{proof}

\begin{rem}
A related result appears in \cite{Neelon}, whose main theorem deals with functions $ g $ belonging to a quasianalytic class on a given interval $ I $ of $ \mathbb{R} $. If, for some real $ \alpha>0 $, the function $ g^\alpha $ is smooth, it is actually in the same class as $ g $ on $ I $. In our local setting, this result can be obtained either by elementary means (writing $ g(x)=x^\nu f(x) $ for some integer $\nu $ and some $ f $ in $\mathcal{E}_1(M) $ with $ f(0)\neq 0$), or as an application of Theorem \ref{main} with $ h=g^\alpha $ and $ \varphi(x,y)=y^q-(g(x))^p $, where $ \alpha=p/q $ (indeed, the smoothness of $ g^\alpha $ implies that $ \alpha $ is rational). The methods of \cite{Neelon} are based on the evaluation of composites of $ g $; they do not apply to roots of general Weierstrass polynomials.  
\end{rem}

\begin{rem}
We still do not know whether Theorem \ref{main} still holds in the higher-dimensional case, that is, when $ \mathcal{E}_1(M) $ is replaced by $ \mathcal{E}_n(M) $ with $n\geq 2 $ and $ h $ is a $ C^\infty $ function germ at the origin in $ \mathbb{R}^n $.
\end{rem}


\subsection{A counter-example}\label{contrex}

We now give the example of a non-quasianalytic local ring $ \mathcal{E}_1(M) $, stable under derivation, and of an element $ g $ of $ \mathcal{E}_1(M) $ whose square root is smooth but does not belong to $ \mathcal{E}_1(M) $. Setting $ \varphi(x,y)=y^2-g(x) $, we therefore see that $ h=\sqrt{g} $ provides a counter-example to the conclusion of Theorem \ref{main} when the quasianalyticity assumption is omitted. Incidentally, this  also answers a question about the roots of flat functions raised in \cite{Neelon}, p.132. 

We proceed as follows. With any real number $ \lambda>0 $, we associate the sequence $ M^\lambda $ defined by $ M_j^\lambda=\exp\bigl(\frac{\lambda}{4}j^2\bigr) $ and the function $ g_\lambda $ defined on $ \mathbb{R} $ by $ g_\lambda(x)=\exp\bigl(-\frac{1}{\lambda}(\ln x)^2\bigr) $ for $ x> 0 $ and $ g_\lambda(x)=0 $ for $ x\leq 0 $. It is easy to check that the sequences $ M^\lambda $ are logarithmically convex and that the local rings $ \mathcal{E}_n(M^\lambda) $ are stable under derivation and non-quasianalytic. Setting $ g=g_\lambda $, the desired counter-example will then be an immediate consequence of the following lemma, since the square root of $ g_\lambda $ is $ g_{2\lambda} $.

\begin{lem}\label{glam} 
The function $ g_\lambda $ belongs to $ \mathcal{E}_1(M^\lambda) $ but not to $ \mathcal{E}_1(M^{\mu}) $ for $ \mu<\lambda $.
\end{lem}

\begin{proof} Clearly, $ g_\lambda $ is $ C^\infty $ on $ \mathbb{R} $ and $ g_\lambda^{(j)}(0)=0 $ for all $ j\geq 0 $. We still denote by $ \ln $ the determination of the logarithm in $ \mathbb{C}\setminus ]-\infty,0[ $ and we set $ G_\lambda(z)=\exp\bigl(-\frac{1}{\lambda}(\ln z)^2\bigr) $. For $ 0<x<1 $, we estimate $g_\lambda^{(j)}(x)$ by applying the Cauchy formula to $ G_\lambda $ on the closed disc $ D_x =\{z\in\mathbb{C} : \vert z-x\vert \leq x/2\} $. Observe that for any $ z\in D_x $, we have 
$ \vert \ln z-\ln x \vert\leq \frac{2}{x}\vert z-x\vert\leq 1 $, hence $ \vert (\ln z)^2-(\ln x)^2\vert\leq \vert \ln z\vert +\vert\ln x\vert\leq 2\vert \ln x\vert +1 $ and $\Re (\ln z)^2\geq (\ln x)^2 +2\ln x-1 $. Thus, we get  
$ \vert G_\lambda(z)\vert \leq \exp(-\frac{1}{\lambda}\bigl((\ln x)^2+2\ln x-1\bigr)\bigr) $ and 
the Cauchy formula yields $ \vert g_\lambda^{(j)}(x)\vert\leq  2^j j!
\exp\Big(-\frac{1}{\lambda}(\ln x)^2 -(j+\frac{2}{\lambda})\ln x+\frac{1}{ \lambda}\Big) $. The maximal value of the right-hand side is obtained for $ \ln x = - \frac{\lambda}{2}(j+\frac{2}{\lambda}) $, hence $  \vert g_\lambda^{(j)}(x)\vert\leq C \sigma^j j! M_j^\lambda $ with $ C=e^{1/\lambda} $ and $ \sigma=2e^{1/4} $. The same estimate is trivial for $ x<0 $ since $ g_\lambda $ vanishes there. Thus, we have $ g_\lambda\in \mathcal{E}_1(M^\lambda) $.

We now prove the second part of the statement. Since $ g_\lambda $ is flat at $ 0 $, the Taylor formula implies 
that for any real $ \delta>0 $, any integer $ j\geq 0 $ and any $ x\in ]0,\delta[ $,  we have 
$ \frac{j!}{x^j} g_\lambda(x)\leq \sup_{0<t<\delta}\vert g_\lambda^{(j)}(t)\vert $. Being given $ \delta $, we have $ \exp\bigl(-\frac{\lambda}{2}j\bigr)\in ]0,\delta[ $  for $ j $ large enough; applying the preceding estimate with $ x=\exp\bigl(-\frac{\lambda}{2}j\bigr) $ then yields 
$  \sup_{0<t<\delta}\vert g_\lambda^{(j)}(t)\vert\geq j!\exp(\frac{\lambda}{4}j^2\bigr)
$. Hence $ g_\lambda $ does not belong to $ \mathcal{E}_1(M^\mu) $ for $ \mu<\lambda $.
\end{proof}

\section{The Ultraholomorphic Case}\label{Ult}

\subsection{Ultraholomorphic function spaces}

Let $ \gamma $ be a positive real number and let $ r $ be either a positive real number or $ +\infty $. We consider the plane sector 
\begin{equation*}
S_{\gamma,r}
=\Big\{ z\in\Sigma\ ;\ \vert \Arg z\vert<\gamma\frac{\pi}{2}\textrm{ and }\vert z\vert <r\Big\},
\end{equation*}
where $ \Sigma $ denotes the Riemann surface of the logarithm and $ \Arg $ the principal value of the argument on $ \Sigma $. We denote by $\mathcal{A}^\infty(S_{\gamma,r})$ the space of holomorphic functions which are bounded, as well as all their derivatives, in $ S_{\gamma,r} $. Any element $ f $ of $ \mathcal{A}^\infty(S_{\gamma,r}) $ has an asymptotic expansion $ \widehat{f} $ at the vertex. Indeed, for $ \gamma<2 $, the sector $ S_{\gamma,r} $ is an open subset of the complex plane and it is easy to see that $ f $ extends continuously, together with all its derivatives, to the closure of $ S_{\gamma,r} $. In particular, $ f $ has a Taylor series at $ 0 $. For $ \gamma\geq 2 $, in other words for sectors on the Riemann surface $ \Sigma $, the Taylor series still makes sense, since all the restrictions of $ f $ to subsectors of $ S_{\gamma,r} $ of aperture smaller than $ 2\pi $ necessarily have the same expansion at $ 0 $.

Being given a real sequence $ M $ satisfying \eqref{norm} and \eqref{logc}, we now denote by $ \mathcal{A}_M(S_{\gamma,r}) $ the subspace of $ \mathcal{A}^\infty(S_{\gamma,r}) $ given by those $ f $ for which there are constants $ C $ and $ \sigma $ such that 
\begin{equation*}
\vert f^{(j)}(z)\vert\leq C\sigma^j j!M_j\text{ for any }j \in\mathbb{N}\text{ and any }z\in S_{\gamma,r}.
\end{equation*}
Obviously, when $ f $ belongs to $ \mathcal{A}_M(S_{\gamma,r})$, its asymptotic expansion $ \widehat{f} $ belongs to the set $\mathcal{F}_1(M) $ of all formal power series $ F=\sum_{j\geq 0}F_jx^j $ for which one can find constants $ C>0 $ and $ \sigma >0 $ such that 
\begin{equation*}
\vert F_j\vert\leq C\sigma^j M_j \text{ for any } j\in \mathbb{N}.
\end{equation*}
For details on $ \mathcal{F}_1(M) $ and its extensions to several variables, we refer the reader to \cite{Th} and the references therein. 

An element $ f $ of $ \mathcal{A}^\infty(S_{\gamma,r}) $ is said to be \emph{flat} if it satisfies $ \widehat{f}=0 $. The ultraholomorphic class $ \mathcal{A}_M(S_{\gamma,r}) $ is said to be \emph{quasianalytic} if the only flat element that it contains is $ 0 $. By a classical result of Korenblum \cite{Kor}, this is the case if and only if 
\begin{equation}\label{kor}
\sum_{j\geq 0}\left(\frac{M_j}{(j+1)M_{j+1}}\right)^\frac{1}{1+\gamma}=\infty. 
\end{equation}

From now on, we shall say that the sequence $ M $ is \emph{strongly regular} if it  satisfies the conditions \eqref{norm}, \eqref{logc}, \eqref{modg} and \eqref{snqa} of Section \ref{spos}. In this case, it is shown in \cite{Thsect} that one can find a positive number $ \gamma(M) $ such that the following Borel-Ritt type extension property holds.

\begin{prop}[\cite{Thsect}, Theorem 3.2.1]\label{Britt}
Assume that $ M $ is strongly regular and let $ F $ be an element of $ \mathcal{F}_1(M) $. Then for any $ \gamma<\gamma(M) $, there is a function $ f $ of $ \mathcal{A}_M(S_{\gamma,\infty}) $ such that $ \widehat{f}=F $.
\end{prop}

The article \cite{Thsect} also provides a way to compute the optimal value of $ \gamma(M) $, but we do not need this precision in what follows. However, it should be mentioned that the condition $ \gamma<\gamma(M) $ implies that \eqref{kor} fails, hence that $ \mathcal{A}_M(S_{\gamma,r}) $ is non-quasianalytic: see \cite{Thsect} for details.

The following  proposition will be useful to relate flatness and regularity properties of elements of $ \mathcal{A}^\infty(S_{\gamma,r}) $. 

\begin{prop}\label{flatness} If $ f $ is an element of $ \mathcal{A}_M(S_{\gamma,r}) $ such that $\widehat{f}=0 $, then there are positive constants $ c_1 $ and $ c_2 $ such that 
\begin{equation}\label{flath}
\vert f(z)\vert \leq c_1 h_M(c_2\vert z\vert)\text{ for any } z\in S_{\gamma,r}.
\end{equation}
Conversely, if $ f $ is an element of $ \mathcal{A}^\infty(S_{\gamma,r}) $ such that \eqref{flath} holds for some constants $ c_1 $ and $ c_2 $, then $ f$ belongs to $ \mathcal{A}_M(S_{\gamma',r'}) $ for any $ \gamma'<\gamma $ and $ r'<r $, and we have $ \widehat{f}=0 $.
\end{prop}

\begin{proof} It is more or less standard, but we briefly recall the argument for the reader's convenience. Notice that it only requires, in fact, the basic properties \eqref{norm} and \eqref{logc} of $ M $. The first part of the statement is a consequence of the Taylor formula. Indeed, the assumption $ \widehat{f}=0 $ implies that for any $ z\in S_{\gamma,r} $ and any integer $ j\geq 0 $, we have
$ \vert f(z)\vert \leq \sup_{0<t<1}\vert f^{(j)}(z)\vert \frac{\vert z\vert^j}{j!}\leq C\sigma^j M_j\vert z\vert^j $, where the constants $ C $ and $ \sigma $ are associated with $ f $ by the definition of $ \mathcal{A}_M(S_{\gamma,r}) $. Taking the infimum with respect to $ j $, we then obtain \eqref{flath} with $ c_1=C $ and $ c_2=\sigma $.
To prove the converse part, we first observe that for any $ \gamma'<\gamma $ and $ r'<r $, there is a real number $ \delta>0 $ such that for any $ z\in S_{\gamma',r'} $, the closed disc $ D_z=\{w: \vert w-z\vert\leq \delta \vert z\vert\} $ lies in $ S_{\gamma,r} $. Using the Cauchy formula on $ D_z $ and \eqref{flath}, we get
$ \vert f^{(j)}(z)\vert\leq \frac{c_1 j!}{\delta^j\vert z\vert^j}h_M(c_3\vert z\vert) $ for any $ j\geq 0 $, with $ c_3=c_2(1+\delta) $. Since the definition of $ h_M $ obviously implies $ h_M(c_3\vert z\vert)\leq (c_3\vert z\vert)^j M_j $, we derive that $ f $ belongs to $ \mathcal{A}_M(S_{\gamma',r'}) $. From the inequality $ h_M(c_3\vert z\vert)\leq (c_3\vert z\vert)^{j+1} M_{j+1} $, we also derive that $ f^{(j)}(z) $ tends to $ 0 $ at the vertex, hence $ \widehat{f}=0 $. 
\end{proof}

\subsection{Smooth roots of ultraholomorphic polynomials}

Our second main result can now be stated as follows.

\begin{thm}\label{main2} Let $ \gamma $ and $ r $ be positive real numbers. Assume that the sequence $ M $ is strongly regular, and
consider a polynomial 
\begin{equation*}
\varphi(z,w)=w^d+a_1(z)w^{d-1}+\cdots+a_d(z)
\end{equation*}
with $ a_j\in\mathcal{A}_M(S_{\gamma,r}) $ for $ j=1,\ldots,d $. If $ h $ is  an element of $ \mathcal{A}^\infty(S_{\gamma,r}) $ such that $ \varphi(z,h(z))=0 $ for $ z\in S_{\gamma,r} $, then $ h $ belongs to $ \mathcal{A}_M(S_{\gamma',r'}) $ for any $ \gamma'<\gamma $ and $ r'<r $.
\end{thm}

\begin{proof}
We remark that the asymptotic expansion $ \widehat{h} $ obviously satisfies $ \widehat{\varphi}(z,\widehat{h}(z))=0 $, where $ \widehat{\varphi}(z,w)=w^d+\widehat{a_1}(z)w^{d-1}+\cdots+\widehat{a_d}(z) $.
Puiseux's theorem in the formal case implies that the equation 
$ \widehat{\varphi}(z,H(z))=0 $, where $ H $ is a formal power series, has a finite number of solutions $ H_1,\ldots,H_K $. Using a version of Artin's theorem for equations in $ \mathcal{F}_1(M)[w] $ (see e.g. \cite{Mouze}) instead of Proposition \ref{Artin}, we can proceed as in the proof of Theorem \ref{main} to derive that each $ H_k $ belongs to $\mathcal{F}_1(M) $. Since $ \widehat{h}=H_k $ for some $ k $, we have in particular
\begin{equation}\label{hinF1M}
\widehat{h}\in \mathcal{F}_1(M).
\end{equation}

Using, if necessary, a finite subdivision of $ S_{\gamma,r} $ into smaller subsectors, and rotating the subsectors, the problem is reduced to the case $\gamma<\gamma(M) $. Thus, by \eqref{hinF1M} and Proposition \ref{Britt}, there is an element $ h_0 $ of $ \mathcal{A}_M(S_{\gamma,r}) $ such that $ \widehat{h_0}=\widehat{h} $. Replacing $ \varphi(z,w) $ by $\varphi(z,h_0(z)+w) $ and $ h $ by $ h-h_0 $, we reduce the problem to the special situation 
\begin{equation}\label{flat} 
\widehat{h}=0.
\end{equation}
In particular, we then have $ \widehat{\varphi}(z,0)=0 $. Thus, there is an integer $ m $, with $ 1\leq m\leq d $, such that
\begin{equation}\label{nulcoef}
\widehat{a_j}=0 \quad\text{for}\quad d-m<j\leq d
\end{equation}
and
\begin{equation}\label{nonul}
\widehat{a_{d-m}}\neq 0, 
\end{equation}
using the convention $ a_0=1 $ in the case $ m=d $. Now, put $ p(z,w)=w^{d-m}+{a_1}(z)w^{d-1-m}+\cdots+{a_{d-m}}(z) $ and $ \Phi(z,w)=\varphi(z,w)-w^mp(z,w) $. 
We then have
\begin{equation}\label{sol}
\Phi(z,h(z))=\varphi(z,h(z))-(h(z))^mp(z,h(z))=(h(z))^m p(z,h(z)). 
\end{equation}
Moreover, by \eqref{flat} and \eqref{nonul}, there are constants $ c_0>0 $ and $ \nu\geq 1 $ such that
\begin{equation}\label{below}
\vert p(z,h(z))\vert\geq c_0\vert z\vert^\nu \text{ for any } z\in S_{\gamma,r}.
\end{equation}
Notice that we also have
$ \Phi(z,w)=\sum_{j=d-m+1}^da_j(z)w^{d-j}. $ Thus, \eqref{nulcoef} and Proposition \ref{flatness} imply that, given a positive real number $ C $, there are constants $ c_1 $ and $ c_2 $, depending only on $C $, $ M $ and $ d $, such that  
\begin{equation}\label{above}
\vert \Phi(z,w)\vert\leq c_1 h_M(c_2\vert z\vert)\text{ for any }z\in S_{\gamma,r}\text{ and any } w \text{ with } \vert w\vert\leq C.
\end{equation}
Gathering \eqref{sol}, \eqref{below} and \eqref{above}, we get
\begin{equation}\label{fin1}
\vert h(z)\vert^m\leq c_0^{-1}c_1\vert z\vert^{-\nu}h_M(c_2\vert z\vert)\text{ for any } z\in S_{\gamma,r}.
\end{equation}
Using properties \eqref{hfunct1} and \eqref{hfunct2}, and up to changing the values of $ c_1 $ and $ c_2 $, we derive
\begin{equation}\label{fin2}
\vert h(z)\vert\leq c_1 h_M(c_2\vert z\vert)\text{ for any } z\in S_{\gamma,r}.
\end{equation}
Proposition \ref{flatness} then yields the desired conclusion.
\end{proof}

\subsection{A question} 
In the preceding proof, property \eqref{hfunct2} allows us to derive the key estimate \eqref{fin2} from \eqref{fin1}. Since \eqref{hfunct2} is itself a consequence of the moderate growth assumption \eqref{modg}, this suggests that \eqref{modg} plays a key role. It actually turns out that it cannot be omitted, as we shall now explain. For any given real number $\lambda>0 $, notice first that the sequence $ M^\lambda $ defined in Section \ref{contrex} satisfies all the requirements of strongly regular sequences, except \eqref{modg}. Beside this, a slight modification of the proof of Lemma \ref{glam} shows that the function $ G_\lambda $ defined there belongs to $ \mathcal{A}_{M^\lambda}(S_{\gamma,r}) $ for $ \gamma $ small enough, but not to $ \mathcal{A}_{M^\mu}(S_{\gamma,r}) $ for $ \mu<\lambda $. Thus, if we put $ \varphi(z,w)=w^2-G_\lambda(z) $ and $ h(z)= G_{2\lambda}(z) $, we obtain a counter-example to the conclusion of Theorem \ref{main2} when the moderate growth assumption is omitted. A natural question now arises. 

\begin{prob}
Does Theorem \ref{main} hold for strongly regular sequences, just as Theorem \ref{main2}? In particular, is the conclusion of Theorem \ref{main} true for Gevrey classes?
\end{prob} 

It should be remarked that the higher-dimensional analogue of this question (that is, when $ h $ and the coefficients $ a_j $ are defined in $ \mathbb{R}^n $ with $ n\geq 2 $) has a negative answer: indeed, Section 4.5 of \cite{Thi0} provides the example of a function $ g $ belonging to a given Gevrey class in $ {\bf R}^2 $ and whose square root is smooth but does not belong to the same class.

\end{document}